\documentclass[12pt]{article}
\usepackage{a4wide}
\newcommand{\version}{April 26, 2016}   

\usepackage{amsthm,amsfonts,amsmath, amscd}
\usepackage{amssymb,amsmath, dsfont}
\usepackage{graphicx}
\swapnumbers
\theoremstyle{plain}

\newtheorem{thm}{THEOREM}[section]

\newtheorem{cl}[thm]{COROLLARY}
\newtheorem{prop}[thm]{PROPOSITION}
\newtheorem{conjecture}[thm]{Conjecture}
\theoremstyle{definition}
\newtheorem{defi}[thm]{DEFINITION}
\theoremstyle{definition}
\newtheorem{remark}[thm]{Remark}
\newcommand{\upchi}{\raise1pt\hbox{$\chi$}}

\newcommand{\C}{{\mathord{\mathbb C}}}

\newcommand{\cH}{{\mathcal{H} }}

\newcommand{\bk}{{\rangle\langle}}

\newcommand{\tr}{{\rm Tr}}

\renewcommand{\|}{{\Vert}}
\renewcommand{\bk}{{\rangle\langle}}
\newcommand{\ef}{{\rm E}_{\rm f}}
\newcommand{\es}{{\rm E}_{\rm sq}}
\numberwithin{equation}{section}
\def\H{\mathcal{H}}

\def\tr{{\rm Tr}}

\usepackage{hyperref}

\begin{document}

\title{Entropy and entanglement bounds for reduced density matrices of
fermionic states}
\author{\vspace{5pt} Eric A. Carlen$^1$, Elliott H. Lieb$^{2}$ and Robin Reuvers$^{3}$\\
\vspace{5pt}\small{$1.$ Department of Mathematics, Hill Center,}\\[-6pt]
\small{Rutgers University,
110 Frelinghuysen Road
Piscataway NJ 08854-8019 USA}\\
\vspace{5pt}\small{$2.$ Departments of Mathematics and
Physics, Jadwin
Hall,} \\[-6pt]
\small{Princeton University, Washington Road, Princeton, NJ
  08544-0001}\\
\vspace{5pt}\small{$3.$ Department of Physics, Jadwin
Hall,} \\[-6pt]
\small{Princeton University, Washington Road, Princeton, NJ
  08544-0001}\\
}
\date{\version}
\maketitle 
\footnotetext                                                                         
[1]{Work partially
supported by U.S. National Science Foundation
grant DMS 1201354.   }  

\footnotetext
[2]{Work partially
supported by U.S. National Science Foundation
grant PHY 0965859  and PHY-1265118.}

\footnotetext
[3]{On leave from Department of Mathematical Sciences, University of Copenhagen. Work supported by ERC Grant Agreement Nos. 321029 and 337603.\\
\copyright\, 2016 by the authors. This paper may be
reproduced, in its
entirety, for non-commercial purposes.}

\begin{abstract}
Unlike bosons, fermions always have a non-trivial entanglement.
Intuitively, Slater determinantal states should be the least entangled
states. To make this intuition precise
we investigate  entropy and entanglement of fermionic
states and prove
some extremal and near extremal properties of reduced density matrices of
Slater determinantal states. 

\end{abstract}

\medskip
\leftline{\footnotesize{\qquad Mathematics subject
classification numbers: 81V99, 82B10, 94A17}}
\leftline{\footnotesize{\qquad Key Words: Density matrix, Entropy, Partial
trace, Entanglement }}

\section{Introduction} \label{intro}

While bosons are often thought of as more complicated than fermions
because of the phenomenon of Bose-Einstein condensation, there is at least
one way in which fermions are more complicated, and  that is the topic
studied here. We investigate the entanglement of fermions caused by
the Pauli principle and we seek the minimum possible entropy and
entanglement. This entanglement is not zero and has been the subject of
much discussion in the literature \cite{AFOV}, but  many questions
about the quantification of entanglement forced by statistics remain open. 
Intuitively, 
minimal entanglement should occur for (Slater) determinantal states. 
To make sense of this intuition, quantitative measurements of entanglement
are needed.

Our motivation for studying minimal fermionic entanglement and minimal
entropy  stems from an effort to understand the two-particle
density matrix of fermions, which is an  important topic for density
functional theory and many-body theory. Information about two-particle
correlations
are potentially useful,  and entanglement is one interesting variety of correlation.

It is important to be clear about the definition of entanglement.
Several authors define entanglement for fermions as the entanglement
relative to a Slater determinantal state. See \cite{AFOV,E,Per} for 
reviews and discussion. This
difference is sometimes called
`correlation'. In any case, whichever definition one uses, a basic
question is to quantify the entanglement of a given multi-particle state
and to determine which states minimize various measures of entanglement.
The next section
will contain precise definitions. While there is a considerable literature
on the definition and information theoretic value of various measures of
entanglement, it is not our goal here to dwell on this matter. Rather, our
goal is to
find precise, sharp values for several measures of entanglement and entropy
for fermionic states, about which relatively little is known.\\

Bosons, in contrast to fermions, are not necessarily entangled. The simplest
state of
$N$ bosons is a condensate, namely the pure state $|\Psi\rangle\langle
\Psi|$
where $\Psi$ is a simple product, $\Psi (x_1,\dots, x_N) = \phi(x_1)\cdots
\phi(x_N)$ and $\phi$ is normalized. All reduced density matrices are of
the same form, i.e., pure products, and all have zero entropy and
entanglement. Of course a bosonic state can have arbitrarily high entropy
and arbitrarily entangled reduced density matrices,
but we are interested in the lowest values. The product states are also
important because they form a (necessarily non-orthogonal) basis for 
bosonic
wave functions. 

The simplest fermionic wave functions one can think of are Slater determinants. 
These span the space and an orthonormal basis can be chosen from among
them.  Our goal is to
show that they have the smallest entropy and entanglement -- and these
minimum values are not zero.   From this point of view, {\em determinants
are  the fermionic analogue of bosonic condensates}. Our first theorem
considers the mutual information of bipartite fermionic states, which is
minimized for 2-particle reduced density matrices of Slater determinants.\\

Let us make a remark before proceeding. The number $N$
will appear in some theorems and it might be argued that each  electron is 
necessarily entangled with all the electrons in the universe through the
Pauli principle, and thus $N\approx \infty$. This is not physically
correct, of course. The mathematical solution to this apparent paradox is to
realize that a density
matrix represents a state on an algebra of observables, and that one must
use the lowest dimension possible to accommodate all the observables under
consideration in the algebra.  One can call this a kind of `coarse
graining'. In our 
case we imagine that $N$ particles are trapped in a box (or in an
isolated atom)  and our observables
refer only to properties inside the box. The Hilbert space for the
$N$-body density matrix is then the
antisymmetric product 
$\wedge^N \mathcal{H}$, where $  \mathcal{H}$ is, e.g., the Hilbert space of
one particle (including spin) in the box. All our theorems apply both to
finite-dimensional Hilbert spaces and to separable infinite-dimensional
ones, but we mention only the finite-dimensional cases for
simplicity. \\

The set-up of the paper is as follows. In Section \ref{mutinf}, we focus on
mutual information and prove a sharp lower bound on this quantity in
Theorem \ref{mutual2}, which is saturated only for Slater determinants.
It is proved using a quantitative
subadditivity inequality (Theorem~\ref{quant}). We then focus on other
measures of entanglement that one might expect are also minimized for
Slaters. We prove sharp bounds that are suggested by these minimization
problems and formulate a number of conjectures that are related to an old
conjecture of Yang concerning the largest possible eigenvalues of fermionic
reduced density matrices, Conjecture \ref{cy}. We provide some new
information on this in Corollary \ref{cycl}. Section \ref{effms} analyzes
entanglement of formation and squashed entanglement, which are relevant
entanglement measures for mixed fermionic states. Section
\ref{proofthms} contains the proofs of a number of theorems, including
Theorems \ref{mutual2} and \ref{quant}.\\

Finally, a word on the conventions we use.
In many-body physics it is
customary 
for the $K$-particle reduced matrix to have trace $\tbinom{N}{K}$. 
Standard
measures of entanglement involve 
entropy, however, and to  define
entropy we require all density
matrices $\rho$ to have {\em unit trace}
($\tr \rho =1$) -- as we do here. 
We also assume that the wedge product of two normalized vectors $\phi$ and $\psi$ has length $\sqrt{2}$, that is, $\phi\wedge\psi:=\phi\otimes\psi-\psi\otimes\phi$. For a wedge product of $K$ vectors, this length is $\sqrt{K!}$. 
Fermionic projectors, generally denoted by $P$ in this paper, satisfy $P^2=P$, so that for example for two particles $P(\phi\otimes\psi)=\phi\wedge\psi/2$. 
These conventions should be kept in mind in what follows.

\section{Mutual information}
\label{mutinf}
Let $\rho_{1\cdots N}$ be a permutation invariant density matrix, as is the
case for either bosons or fermions. The $K$-particle reduced density matrix
is defined to be
\[
 \gamma_K = \tr_{K+1 \cdots N} \ \rho_{1\cdots N}.
\]
where $\tr_{K+1 \cdots N}$ denotes the partial trace over the last $N-K$
factors of $\cH$.

A permutation invariant state $\rho_{1\cdots N}$ may be regarded as a
bipartite state in
$N-1$ different ways, corresponding to the factorizations 
$(\otimes^K\cH)\otimes (\otimes^{N-K}\cH)$ for $K =1,\dots,N-1$.  
  In what follows, we only consider permutation invariant
$N$-particle states. 

When  $\rho_{1\cdots N}=  |\psi\bk \psi|$ is a pure state, the entanglement of $ |\psi\bk \psi|$ regarded
as a bipartite state on $(\otimes^K\cH)\otimes (\otimes^{N-K}\cH)$ is
naturally quantified \cite{B1} as
\begin{equation}\label{entropyq}
S(\gamma_K) = S( \gamma_{N-K})\ ,
\end{equation}
while for mixed states, the issue of quantifying entanglement is more complicated \cite{B2}. We return to this later, and focus for
now on the entanglement of $N$-particle fermionic states as measured by the
quantities in (\ref{entropyq}).  It is natural to seek sharp lower bounds on
the entropies in 
(\ref{entropyq}) for $N$-particle fermionic states, which by the concavity of the entropy, will be minimized by pure states.

The only case that is clearly understood  is that in which $K=1$ (and by
symmetry $K = N-1$).
Coleman's Theorem (see
e.g., \cite[Theorem 3.1]{ls}),  says that
the extreme points of the set of all reduced one-particle density matrices of
$N$-particle (mixed) fermionic state are the reduced density matrices of 
$N$-particle Slater determinants, for which all eigenvalues are exactly
$1/N$. That implies that
the largest eigenvalue of 
$\gamma_1$ is at most $1/N$. For any $M$-dimensional density matrix $\rho$ with eigenvalues $\{\lambda_1,\dots,\lambda_M\}$,
$$S(\rho) = -\sum_{j=1}^M \lambda_j \ln(\lambda_j) \geq   -\sum_{j=1}^M \lambda_j\ln\|\rho\|_\infty  =  -\ln \|\rho\|_\infty,$$
with equality if and only if all positive eigenvalues are equal. 

It follows that if $\rho_{1\dots N}$ is an $N$-particle 
fermionic  state, then $S(\gamma_1) \geq \ln N$, and there is equality if and only if $\rho_{1\dots N}$ is an $N$-particle Slater. 

This settles the cases $K=1$ and $K=N-1$, but since we lack an analogue of
Coleman's Theorem for other values of $K$, there is no easy route to a lower
bound even for $S(\gamma_2)$, although, as we discuss below, it is likely
that the lower bound
is again given by Slaters. 
It will be useful to consider the entropies in (\ref{entropyq}) as measures of mutual information.

\begin{defi} Let $\rho_{1\cdots l}$ be a density
matrix on a Hilbert space $\cH_1\otimes\dots\otimes\cH_l$, and let
$\rho_j$ 
be the reduced density matrix on $\H_j$. The {\em mutual information} in the
state $\rho_{1\cdots l}$ is the quantity
\begin{equation}
\label{mutual}
\sum_{j=1}^l S(\rho_j) - S(\rho_{1\cdots l})\ .
\end{equation}
\end{defi}

The difference in (\ref{mutual})  is well known to be non-negative and zero
if and only if $\rho_{1\cdots N} = \rho_1\otimes\dots\otimes\rho_l$. When
$\rho_{1\cdots N}$ is fermionic, this last condition is impossible, and it
is natural to seek optimal lower bound
on the mutual information for fermionic states.

The entropies in (\ref{entropyq}) can be expressed in terms of mutual informations because if $\rho_{1\cdots N}$ is a pure fermionic $N$-particle 
state regarded as a bipartite state on $(\otimes^K \cH ) \otimes
(\otimes^{N-K}\cH)$, then the mutual information is exactly
$2S(\gamma_K) = 2S( \gamma_{N-K})$.

\begin{thm}[Mutual Information lower bounds]\label{mutual2}  
Let $\gamma_1,\gamma_2$ be the
reduced $1,2$-particle
density matrices of an $N$-particle fermionic state, respectively. 
Then
\begin{equation}\label{subadd1}
2S(\gamma_1) -S(\gamma_2)\geq \ln\left(\frac{2}{1 -
\tr\gamma_1^2}\right)\ ,
\end{equation} 
and there is equality if and only if the $N$-particle fermionic state is a
pure-state Slater determinant.

More generally, for an $N$-particle fermionic state $\rho_{1\dots N}$
\begin{equation}\label{subedn}
 NS(\gamma_1) - S(\rho_{1\dots N})  \geq -\ln
e_N(\gamma_1),
\end{equation}
where $e_N (\gamma_1) $ is the $N^{\rm th}$ elementary symmetric function
of the eigenvalues of $\gamma_1$, namely, $\sum_{i_1<i_2< ...<i_N}
\lambda_{i_1} \cdots \lambda_{i_N}$. This can be expressed in terms of
$p_j = \tr \gamma_1^j$ as

\begin{equation}\label{power}
e_N = \frac{1}{N!}\, \det \left[\begin{array}{ccccc}
1 & 1 & 0 & \dots &\phantom{0}\\
p_2  & 1 & 2 & 0 &\dots\\
\vdots & &\ddots &\ddots & \\
p_{N-1} & p_{N-2} & \dots & 1 & N-1\\
p_N & p_{N-1} & \dots & p_2 & 1\end{array}\right]
\end{equation}
For example, $e_3 =  (1-3p_2 +2p_3)/6$.

\end{thm}

\begin{remark}There is equality in (\ref{subadd1})  when $\gamma_2$
is is the reduced $2$-particle density matrix of an $N$-particle Slater
determinant since, in this case,
$$S(\gamma_2) = \ln \tbinom{N}{2}\ , \quad S(\gamma_1)= \ln N,\quad  {\rm
and}\quad \tr\gamma_1^2 = 1/N.$$

\end{remark}

We now apply this when $\rho_{1\dots N} = |\psi\bk \psi|$ is an $N$-particle 
fermionic pure state, and seek to estimate the entropies in (\ref{entropyq}). By Jensen's inequality,
\begin{equation}\label{jensen}
e^{-S(\gamma_1)} \leq \tr\gamma_1^2 .
\end{equation}
Therefore, (\ref{subadd1}) implies a bound that can be expressed entirely in terms of entropy:
\begin{equation}\label{subadd1p}
2S(\gamma_1) -S(\gamma_2)\geq \ln\left(\frac{2}{1 -
e^{-S(\gamma_1)} }\right).
\end{equation} 
If $\gamma_2=\rho_{12}$ is pure $S(\gamma_2) =0$, and we have that $2S(\gamma_1) - \ln(2/(1 - e^{-S(\gamma_1)})) \geq 0$, and this implies that 
$S(\gamma_1) \geq \ln 2$, which also follows from the fact that the largest
eigenvalue of $\gamma_1$ is no greater than $1/2$, but 
this shows that even the weakened form (\ref{subadd1p}) is sharp.

We now return to the problem of estimating the entropies in (\ref{entropyq}).  When $K=1$ or $N-1$
we have seen that Coleman's Theorem provides a complete description of the convex set of reduced one particle 
density matrices of $N$-particle fermionic states, from which it readily follows that:
\textit{1-particle reduced density matrices
of Slater determinants minimize the entropy and maximize both the
Hilbert--Schmidt norm and the largest eigenvalue.}

Less is known about the set of all fermionic reduced $2$-body density matrices. Yang's theorem \cite{Y} says that no eigenvalue of such a density matrix can
exceed
$\frac{N}{2} {\tbinom{N}{2}}^{-1}= 1/(N-1)$, and this is attained for the
so-called Yang pairing state (see 
the start of Section \ref{largesteig} for a definition). This is large compared to
${\tbinom{N}{2}}^{-1}$, which 
is the value of every non-zero eigenvalue for a pure Slater determinant.
Yang's bound would allow a 
reduced density matrix with $(N-1)$ eigenvalues close to this size, and
hence an entropy of order $\ln(N)$. 
However, the 2-particle reduced density matrix of the pairing state has a very large
entropy, which we compute Proposition \ref{eigsandmults},
due to a
large number of very small eigenvalues, and is not competitive in the
search for any entropy minimizer. In contrast, the entropy of a reduced
$2$-particle density matrix of an $N$-particle
Slater state is $\ln\tbinom{N}{2}$, which is of order $2\ln(N)$.
We therefore conjecture: 
\begin{conjecture} 
\label{c1}
The 2-particle reduced density matrix of a
Slater determinant minimizes the entropy, that is,
\[
S(\gamma_2) \geq \ln\tbinom{N}{2},
\]
where $\gamma_2$ is the 2-particle density matrix of any $N$-particle
fermionic state.
\end{conjecture}
By Jensen's inequality \eqref{jensen}, this would be implied by the stronger
conjecture that 2-particle reduced density matrices of Slater determinants maximize the Hilbert--Schmidt norm:
\begin{conjecture}
\label{c2}
\[
\tr[\gamma_2^2] \geq \tbinom{N}{2}^{-1}
\]
\end{conjecture}
The Yang state also has a squared Hilbert--Schmidt norm of order
$N^{-2}$, but it is smaller than the value above. (See Remark~\ref{yangp} below.)

A weaker conjecture is:
\begin{conjecture}
\label{c3}
\begin{equation}
\label{weakestconjecture}
S(\gamma_2) \geq 2\ln N+ \mathcal{O}(1)
\end{equation}
\end{conjecture}

A strategy to prove this conjecture comes from the proof of Theorem
\ref{mutual2}, which is based on a quantitative
subadditivity inequality proved in Theorem~\ref{quant}. Suppose
now that $N = 2n$, and that $\rho_{1\cdots N}$ is a
fermionic
$N$-particle state.  Let $\gamma_2$ denote its reduced $2$-particle density
 matrix.  The relative entropy of $\rho_{1\cdots N}$ with respect to the product state $\otimes^n\gamma_2$
 is precisely $nS(\gamma_2) - S(\rho_{1\cdots N})$ and is non-negative on account of the subadditivity of the entropy.  Theorem~\ref{quant} below provides a lower bound for this non-negative quantity,
 which is
\begin{equation*}\label{st0}
S(\rho_{1\cdots N}) -\tfrac{N}{2}S(\gamma_2) \leq 2\ln\left( \tr\left[
\sqrt{\rho_{1\cdots N}}\left(\sqrt{\gamma_2}\otimes \cdots \otimes
\sqrt{\gamma_2}\right)\right]\right).
\end{equation*}
Since the eigenfunctions of $\sqrt{\rho_{1\cdots N}}$ are antisymmetric, we may replace 
$\sqrt{\gamma_2}\otimes \cdots \otimes
\sqrt{\gamma_2}$ in this estimate by $P\sqrt{\gamma_2}\otimes \cdots \otimes
\sqrt{\gamma_2}\ P$ where $P$ is the projector onto the antisymmetric subspace, and then apply Cauchy--Schwarz to obtain
\begin{equation*}\label{st00}
S(\rho_{1\cdots N}) -\tfrac{N}{2}S(\gamma_2) \leq 
\ln\left(\tr\left[ P\left(\gamma_2\otimes \cdots \otimes
\gamma_2\right)P\right]\right).
\end{equation*}
By concavity of the entropy (or convexity of the Hilbert--Schmidt norm), it
suffices to consider pure $\rho_{1\cdots N}$ to prove the conjectures
above, in which case 
\begin{equation}
\label{entropyest}
S(\gamma_2)\geq-\tfrac{2}{N}\ln\left(\tr \left[P\left(\gamma_2\otimes \cdots
\otimes
\gamma_2\right)P\right]\right).
\end{equation}
To study the norm, we prove the following theorem. 
\begin{thm}
\label{equivconst}
Let $P$ denote the projector onto $\wedge^N\mathcal{H}$ and let $1\leq K\leq
N-1$. Let $M$ be the dimension of the Hilbert space. We define
\begin{equation}
\label{normproblem}
C^{M,N}_{K}\ \
:=\sup_{\substack{\text{$\|\psi_1\|=\|\psi_2\|=1$}\\\text{
$\psi_1\in\wedge^K\mathcal{H}$,
$\psi_2\in\wedge^{N-K}\mathcal{H}$}}}\|P(\psi_1\otimes\psi_2)\|^2\ \ ,
\end{equation}
and, using the convention that $\tr(\gamma_K^\Psi)=1$,
\begin{equation}
\label{eigproblem}
\Lambda^{M,N}_K\ \ :=
\sup_{\substack{\text{$\|\Psi\|=1$}\\\text{$\Psi\in\wedge^N\mathcal{H}$}}}
\lambda^{\text{max}}(\gamma_K^\Psi)\ .
\end{equation}
We then have that
$\Lambda^{M,N}_K=C^{M,N}_{K}=C^{M,N}_{N-K}=\Lambda^{M,N}_{N-K}$.
\end{thm}
\begin{proof}
On the one hand, we have for any normalized $\Psi\in\wedge^N\mathcal{H}$
$$
\lambda^{\text{max}}(\gamma_K^\Psi)\ \
=\sup_{\substack{\text{$\|\psi_1\|=\|\psi_2\|=1$}\\\text{
$\psi_1\in\wedge^K\mathcal{H}$,
$\psi_2\in\wedge^{N-K}\mathcal{H}$}}}\left|\langle\Psi,
\psi_1\otimes\psi_2\rangle\right|^2\ 
\leq\  C^{M,N}_{K}.
$$
The first inequality is obtained by calculating the reduced density matrix
from the Schmidt decomposition of $\Psi$, and the second by applying
Cauchy--Schwarz and by using that $\Psi=P\Psi$. 
On the other hand, any normalized $\psi_1\in\wedge^K\mathcal{H}$ and
$\psi_2\in\wedge^{N-K}\mathcal{H}$ satisfy
$$
\|P(\psi_1\otimes\psi_2)\|^2\ \
=\sup_{\substack{\text{$\|\phi_1\|=\|\phi_2\|=1$}\\\text{
$\phi_1\in\wedge^K\mathcal{H}$,
$\phi_2\in\wedge^{N-K}\mathcal{H}$}}}\left|\left\langle\frac{
P(\psi_1\otimes\psi_2)}{\|P(\psi_1\otimes\psi_2)\|},
\phi_1\otimes\phi_2\right\rangle\right|^2\leq \Lambda^{M,N}_K,
$$
where we again applied the Schmidt decomposition to obtain the final
inequality.
\end{proof}

\begin{cl}
For an $M$-dimensional Hilbert space $\mathcal{H}$, any vectors
$\psi_1\in\otimes^K\mathcal{H}$ and $\psi_2\in\otimes^{N-K}\mathcal{H}$
satisfy
$$
\|P(\psi\otimes\tilde{\psi})\|^2\leq
\Lambda^{M,N}_K\|\psi_1\|^2\|\psi_2\|^2,
$$
and any density matrices $\rho_1$ on
$\otimes^K\mathcal{H}$ and $\rho_2$ on $\otimes^{N-K}\mathcal{H}$ satisfy
$$
\tr\left[P(\rho_1\otimes\rho_2)P\right]\leq
\Lambda^{M,N}_K\tr\left[\rho_1\right]\tr\left[\rho_2\right].
$$
\end{cl}
Note that both results can be iterated to obtain further inequalities on
composite vectors and density matrices. For $K=2$ and $N=2n$ and $M=2m$
even, we have more information from Yang's theorem (see Proposition \ref{Yangresult}):
$\Lambda^{M,N}_2=(N-1)^{-1}(m-n+1)/m$, and the maximizer is unique.

\begin{cl}
\label{2particlecl}
Let $N=2n$ and $M=2m$. Let $\psi_1,\dots,\psi_n\in\mathcal{H}\wedge\cH$ be
normalized. We then have
\[
\|P(\psi_1\otimes\dots\otimes\psi_n)\|\ \leq\ 
\prod^n_{j=1}\ \Lambda^{M,2j}_2=\ \prod^n_{j=1}\ \frac{1}{2j-1}\frac{m-j+1}{m}
\]
with equality if and only if these vectors are all equal and an equal sum of
pairs.
\end{cl}
Returning to \eqref{entropyest}, we note that an application
of Corollary \ref{2particlecl} gives the estimate
\[
S(\gamma_2) \geq \ln N+ \mathcal{O}(1),
\]
which is a factor 2 off from \eqref{weakestconjecture}. Note that this also
follows directly from Yang's theorem and the fact that the entropy is bounded below by $-\ln \|\gamma_2\|_\infty$. This estimate is expected to be far from optimal
because the 2-particle reduced density matrix of an $N$-particle fermionic
state can never be rank 1, which would have be the case to satisfy the
bound in Corollary \ref{2particlecl}.

Motivated by the relevance of the Yang state and largest eigenvalues of reduced density matrices, we
present some calculations in the next section.

\section{The Yang pairing state} 
\label{largesteig}
The Yang pairing state \cite{Y} is defined as follows.  
Let $\H= \C^M$ be the one-particle Hilbert space and consider $N\leq M$ particles. Assume that both $M$ and $N$ are even integers, so that we can define integers $m=M/2$ and
$ n=N/2$.
We choose an orthonormal basis of $\H$, 
$u_i$ with $1\leq i\leq M$ and we consider the set
of $2n$-particle Slater
determinants $\phi_\alpha$  that are composed of  $n$ pairs of vectors $
\pi_i =u_{2i-1},
u_{2i}$. There are $m$ such pairs, which are a small fraction of the
$\tbinom{2m}{2}$ pairs with arbitrarily chosen indices. The number of determinants $\phi_\alpha$ that we can build from these pairs is
$\tbinom{m}{n}$. The pairing state (a vector in $ \H^{\wedge N}$) is given by the equal superposition of these determinants:
\begin{equation}
\label{pstate}
 |\Psi_{M,N}\rangle =    \tbinom{m}{n}^{-1/2} \sum _\alpha \phi_\alpha .
\end{equation}

The eigenvalues of reduced density matrices of the Yang pairing state can
give information about its entropic properties, so we will consider these first.

Yang \cite{Y} proved the following optimality result for the pairing state.
\begin{prop}
\label{Yangresult}
Let $\gamma_2^\Psi$ be the 2-particle reduced density matrix of an
N-particle fermionic state $\Psi\in\wedge^N\C^M$. If $M=2m$ and $N=2n$ are
even
\begin{equation}
\label{Yangproof}
\lambda^{\text{max}}(\gamma_2^\Psi)\leq  \frac{1}{N-1}\frac{m-n+1}{m},
\end{equation}
and this is attained if and only if $\Psi$ is a Yang pairing state in some
basis.

More generally, let $M=2m$ if it
is even and $M=2m+1$ if it is odd. Similarly let $N=2n$ if it is even and
$N=2n+1$ if it is
odd. Then, we have that
\begin{equation}
\label{newproof}
\lambda^{\text{max}}(\gamma_2^\Psi)\leq  \begin{cases}   (N-1)^{-1} &  {\rm
if}\ N\ {\rm is\ even}\\
\ \ \ N^{-1} &  {\rm
if}\ N\ {\rm is\ odd}
\end{cases}\ \ \  .
\end{equation}
\end{prop}

\begin{remark}
Yang's proof of \eqref{Yangproof} uses induction on both $M$ and $N$ and is
rather involved. It turns out that
the $M\to\infty$ behaviour follows from a simple argument, which
generalizes to odd $N$ and $M$. We now give this simple proof of
\eqref{newproof}.
\end{remark}

\begin{proof}
To find the largest possible eigenvalue, we should consider
\[
\sup_{\Psi\in\wedge^N\mathcal{H},\ f\in\mathcal{H}\wedge\mathcal{H}
}(f,\gamma_2^\Psi f).
\]

Given $f\in\mathcal{H}\wedge\mathcal{H}$, a result of Youla \cite{Yo}, and
at about the same time Yang \cite{Y}, states that there are an orthonormal
basis
$\{u_i\}_{1\leq i\leq M}$ of $\cH$
and positive numbers $\{d_j\}_{1\leq j\leq m}$ so that
\begin{equation}
\label{YangYoula}
f = \sum^m_{j=1}
d_j \  \frac1{\sqrt{2}} u_{2j-1}\wedge u_{2j} ,
\end{equation}
where $\sum_{j} (d_j)^2 =1$ and the convention for $\wedge$ was mentioned in the introduction.

Let $\boldsymbol{\alpha}$ denote a set of $N$ indices 
$\{\alpha_1,\dots,\alpha_N\}$ where $1\leq\alpha_k\leq M$, 
and, if these indices are all different, let $u_{\boldsymbol{\alpha}}$
be given by
$$
u_{\boldsymbol{\alpha}}  = u_{\alpha_1} \wedge \cdots \wedge
u_{\alpha_N},
$$
where $\alpha_1\leq\dots\leq\alpha_N$.
We can expand any state $\Psi$ in Slaters built from $\{u_i\}_{1\leq i\leq M}$:
$$ 
|\Psi\rangle=\sum_{\boldsymbol{\alpha}}c_{\boldsymbol{\alpha}}
\frac{1}{\sqrt{N!}}u_{\boldsymbol{\alpha}},
$$
where $c_{\boldsymbol{\alpha}}=0$ if $\boldsymbol{\alpha}$ contains the same
index more than once and $\sum_{\boldsymbol{\alpha}}
|c_{\boldsymbol{\alpha}}|^2 =1$. 
The 2-particle reduced density matrix is then
$$
\gamma^\Psi_2=\tbinom{N}{2}^{-1}\sum_{i<j,\
i'<j'}\left(\sum_{\boldsymbol{\beta}}c_{\{i,j\}\cup\boldsymbol{\beta}}\
\overline{c_{\{i',j'\}\cup\boldsymbol{\beta}}}\
(-1)^{\sigma(i,j,\boldsymbol{\beta})+\sigma(i',j',\boldsymbol{\beta})}
\right)\frac12|u_i\wedge u_j\rangle\langle u_{i'}\wedge u_{j'}|
$$
where $\boldsymbol{\beta}$ is a set of $(N-2)$ indices, and
$\sigma(i,j,\boldsymbol{\beta})$ is the sign of the permutation that orders
$\{i,j\}\cup\boldsymbol{\beta}$.

For $1\leq j\leq m$, let $p_j=\{2j-1,2j\}$ denote a pair of indices
appearing in \eqref{YangYoula}. We have
\begin{equation}
\label{problemmin}
\begin{aligned}
(f,\gamma_2^\Psi
f)&=\tbinom{N}{2}^{-1}\sum_{\boldsymbol{\beta}}\left|\sum_{j}
c_{p_j\cup\boldsymbol{\beta}}\ (-1)^{\sigma(p_j,\boldsymbol{\beta})}\
d_j\right|^2\\
&\leq\tbinom{N}{2}^{-1}\left(\sum_{j} d_j^2\right)
\sum_{\boldsymbol{\beta},j}
|c_{p_j\cup\boldsymbol{\beta}}|^2\leq \tbinom{N}{2}^{-1}n,
\end{aligned}
\end{equation}
where we have used Cauchy--Schwarz, $\sum_{j} (d_j)^2 =1$ and
$\sum_{\boldsymbol{\alpha}} |c_{\boldsymbol{\alpha}}|^2 =1$, and, also the fact
that each $\boldsymbol{\alpha}$ can contain at most $n$ pairs and hence
appear at most $n$ times in the sum above. 
\end{proof}

\begin{remark}
The first line in \eqref{problemmin} is exact, so it should really give the
maximum $\Lambda^{M,N}_2$ in the even case. The missing factor compared to
\eqref{eigsandmults} comes from the fact that there should be no overlap
between the indices in $p_j$ and $\boldsymbol{\beta}$. This causes the
preferred strategy to be to spread the weight evenly across the $d$'s and
$c$'s (as the Yang state does). Unfortunately, we have not been able to find
a simple argument to prove that this is indeed the best strategy. Since
optimization problems for fermions, when written out in
coefficients, are difficult because of constraints imposed by the exclusion
of repeated indices (such as the constraint
$p_j\cap\boldsymbol{\beta}=\emptyset$ in \eqref{problemmin}), it would be
good to understand the reason that the optimizers in \eqref{problemmin}
have uniformly distributed coefficients.
\end{remark}

The following proposition calculates a number of quantities for a Yang pairing state. 

\begin{prop}
\label{Ys}
Let $N=2n$ and $M=2m$ and let $\gamma_2$ be the two-particle reduced
density matrix of the 
N-particle pairing state \eqref{pstate} built on  $\H= \C^M$. Its
eigenvalues are
\begin{equation} 
\label{eigsandmults}
\Lambda^{M,N}_2=\frac{1}{N-1}\frac{m-n+1}{m} , \ \ \qquad 
\lambda^{M,N}_2=\frac{1}{N-1}\frac{n-1}{m(m-1)},\ \  
\end{equation}
where $\Lambda^{M,N}_2$ has multiplicity 1 and $\lambda^{M,N}_2$ has
multiplicity $2m^2-m-1$.
Consequently, the entropy is
\[
\begin{aligned}
S(\gamma_2)&= -\frac{m-n+1}{m(2n-1)}\ln\left[\frac{m-n+1}{m(2n-1)}\right]\\
&\quad-\frac{(n-1)(2m+1)}{m(2n-1)}\ln \left[\frac{(n-1)}{(2n-1)m(m-1)}\right], 
\end{aligned}
\]
and for $p\geq1$
\begin{equation}
\label{pnormrho}
 \tr[\gamma_2^p]=\left(\frac{1}{N-1}\frac{m-n+1}{m}
\right)^p+(2m^2-m-1)\left(\frac{1}{N-1}\frac{n-1}{m(m-1)}\right)^p.
\end{equation}
\end{prop}

\begin{remark}\label{yangp}
Asymptotically, for $M
\gg N \gg 1$,
the leading term is $S(\gamma_2) \asymp  2\ln M$.
Thus, $S$ can be much larger than $O (\ln N)$, as it is
for a determinant, and can even be infinite. 
Although the
pairing state has a larger eigenvalue (asymptotically $1/N$ instead of
$2/N^2$), and potentially a smaller
entropy, it has so many small eigenvalues that its entropy can
be huge.

Maximizing \eqref{pnormrho} in $M$, we find  
\[
\tr[\gamma_2^p]\leq \begin{cases} 
\ \ \left(\frac{1}{n(N-1)}\right)^{p-1} \ \ \ \ \ &{\rm if \ \ } p\leq1+\tfrac{\ln(N-1)}{\ln(n)}\\
\phantom{xx}\\
\ \ \ \ \left(\frac{1}{N-1}\right)^p \ \ \ \ \ &{\rm if \ \ } p\geq1+\tfrac{\ln(N-1)}{\ln(n)}
 \end{cases}.
\]

\end{remark}
\begin{proof}
Let $ \gamma$ be the 2-particle reduced density matrix for the pure
state $| \Psi_{M,N}\rangle \langle \Psi_{M,N}|$, where $\Psi_{M,N}$ is the a pairing state defined in \eqref{pstate}, {\em normalized so that its trace is
$\binom{N}{2}$}. 
(This will simplify a number of expressions below.) Note that Yang \cite{Y}
takes the trace of the 2-particle density matrix to be normalized to
$2\tbinom{N}{2}$.

It is easy to compute
the matrix
elements of $\gamma$ (with $i<j$ and $k< \ell$):
\[
\frac12\tbinom{m}{n}\langle u_i\wedge u_j |\,\gamma\, | u_k \wedge u_\ell
\rangle
 = \begin{cases}    \tbinom{m-1}{n-1} &  {\rm
if}\ i,j = k,\ell \ {\rm is\ a\ pair} \ \pi_i\\
 \tbinom{m-2}{n-1} &  {\rm
if}\ i,j \ {\rm and }\ k,\ell \ {\rm are\ unequal \ pairs} \\
\tbinom{m-2}{n-2} & {\rm
if}\ i,j \ {\rm and }\ k,\ell \ {\rm are\ equal\ and \ not \ pairs} \\
\ 0 & {\rm otherwise}.
\end{cases}
\]
Therefore, the matrix $\gamma$ has the structure
\begin{equation}\label{form}
\gamma =   \frac{n(m-n) }{m-1} |\chi\rangle\langle \chi| +   \frac{ n(n-1)
}{m(m-1)}P_{\H \wedge \H}
\end{equation}
where 
${\displaystyle \chi = \frac{1}{\sqrt{2m}}\textstyle{\sum_{j=1}^m}
u_{2j-1}\wedge u_{2j}}$ is a unit vector,
and where $P_{\H \wedge \H}$ is the orthogonal projection in $\H\otimes \H$
onto $\H\wedge \H$.
Therefore, $\gamma$ will have two different eigenvalues, with
multiplicities $\mu$,  as follows:
\begin{equation}\label{form2}
\frac{n(m-n) }{m-1} + \frac{ n(n-1) }{m(m-1)}, \ \  \mu = 1 ; \qquad 
\frac{n(n-1)}{m(m-1)},\ \  \mu = 2m^2 -m -1. 
\end{equation}
Using either  (\ref{form}) or   (\ref{form2}), one computes that 
$$\tr(\gamma) = \frac{n(m-n) }{m-1} + (2m^2 -m )\frac{n(n-1)}{m(m-1)}  =
2n^2 -n  = \tbinom{N}{2}\ ,$$
as it must be.

Now letting $\gamma_2$ denote the normalized $2$-particle reduced density
matrix (i.e.\ with trace 1), we have 
\begin{equation}\label{form3}
\gamma_2=   \frac{(m-n) }{(2n-1)(m-1)} |\chi\rangle\langle \chi| +  
\frac{ n-1 }{(2n-1)m(m-1)}P_{\H \wedge \H}\ ,
\end{equation}
which gives the stated eigenvalues.
\end{proof}

One might expect that a similar argument (or the more complete proof by Yang) generalizes to $K$-particle reduced density matrices for $K\geq3$. However, both proofs hinge on the Yang--Youla description \eqref{YangYoula} of a fermionic bipartite state, which is particularly clear in the proof above.

The Yang--Youla canonical form for vectors in $\Psi\in\wedge^2\cH\subset\cH\otimes\cH$
follows easily from the variational characterization of the constituents of the Schmidt decomposition of vectors in 
$\cH_1\otimes \cH_2$.
Recall that any such vector has the expansion 
\[
\Psi=\sum_j \sigma_j u_j\otimes v_j,
\]
where $\{\sigma_j\}$ is a non-increasing sequence of non-negative numbers, $\{u_j\}$ is an orthonormal basis of $\cH_1$
and $\{v_j\}$ is an orthonormal basis of $\cH_2$.
In case $\Psi\in \wedge^2\cH$, the variational characterization of $\sigma_1$ gives 
$$\sigma_1 = |\langle \Psi, u_1\otimes v_1\rangle| \geq   \langle \Psi, u\otimes v\rangle|$$
for all unit vectors $u,v \in \cH$. It follows immediately from the antisymmetry of $\Psi$ that $u$ and $v$ are orthogonal,
and that   $|\langle \Psi, v_1\otimes u_1\rangle| = |\langle \Psi, u_1\otimes v_1\rangle|$. Since $v_1\otimes u_1$ is orthogonal to
$u_1\otimes v_1$, it is a valid trial vector for $\sigma_2$, and thus $\sigma_2 = \sigma_1$. Thus
the singular values and the vectors in the Schmidt decomposition come in pairs, and this is precisely the Yang--Youla canonical form \eqref{YangYoula}.

Similarly, any $\Psi\in\wedge^{l_1+l_2}\cH\subset\wedge^{l_1}\cH\otimes\wedge^{l_2}\cH$
can be written in Schmidt form
\[
\Psi=\sum_j \lambda_j \psi_j\otimes \phi_j,
\]
where $\psi_j\in\wedge^{l_1}\cH$ and $\phi_j\in\wedge^{l_2}\cH$.
Again, $\wedge^{l_1}\cH\otimes\wedge^{l_2}\cH$ is much larger than
$\wedge^{l_1+l_2}\cH$. The fact that $\Psi$ is antisymmetric, so that
$P\Psi=\Psi$, imposes conditions on the Schmidt vectors and Schmidt
numbers. While for $l_1=l_2=1$, it is easy to translate these conditions into the
canonical Yang--Youla form, for $l_1+l_2\geq3$, the analogue of this canonical form
(that would presumably involve vectors of the form $\psi\wedge\phi$) is
unknown, which is why there is no equivalent of Yang's theorem for $K\geq3$. 

There is, however, a conjecture by Yang \cite{Y} regarding this, which remains open after more than fifty years: 
\begin{conjecture}[Yang's conjecture]
\label{cy}
There
exist constants $\beta_3, \beta_4,\dots$ such that
\begin{equation}
\label{Yangsconjecture}
\Lambda^{M,N}_K\leq \begin{cases} 
\frac{(N-K)!}{N!}N^{K/2}\beta_K\ \ \ \ \ &{\rm if \ } K \ {\rm
is\ even}\\
\phantom{xx}\\
\frac{(N-K)!}{N!}N^{(K-1)/2}\beta_K \ \ \ \ \ &{\rm if \ } K \ {\rm is\  odd}
 \end{cases}.
\end{equation}
\end{conjecture}
In \eqref{Yangsconjecture}, we assumed that the trace is equal to 1,
rather than the $N!/(N-K)!$ that is used in Yang's paper. This
formulation of the conjecture is somewhat vague because of the
unspecified constants that could be and indeed have to be very large if $K$
is approximately $N/2$, which will be emphasized below. Theorem
\ref{equivconst} allows us to be a bit more precise and give lower bounds on
these constants, simply by plugging the Yang state as a trial function in
the norm problem \eqref{normproblem}, for which it is easy to do
computations.

If $M=2m$, $N=2n$ and $K=2k$, we consider a tensor product of two Yang states built from the same orthonormal basis and find that 
\[
\Lambda^{M,N}_K\geq\|P(\Psi_{M,K}\otimes\Psi_{M,N-K})\|^2=\frac{\tbinom{n}{k
}}{\tbinom{N}{K}} \frac{\tbinom{m-n+k}{k}}{\tbinom{m}{k}},
\]
with $P$ being the antisymmetric projector.
As $M\to\infty$, this tends to 
\begin{equation}
\label{higheig}
\frac{\tbinom{n}{k}}{\tbinom{N}{K}}=\frac{1\cdot3\cdot5\cdot\ \dots\ \cdot(K-1)}{(N-1)(N-3)(N-5)\dots(N-K+1)}.
\end{equation}
For $K\approx N/2$, this gives a largest eigenvalue that seems much higher than
Yang's conjecture \eqref{Yangsconjecture}. There is no contradiction,
however, since the unspecified constants $\beta_K$ can be very large. It is
difficult to say whether there exist fermionic $K$-particle reduced density matrices with higher eigenvalues than \eqref{higheig}; the key to proving a good
upper bound seems to find a canonical Yang--Youla form for $K\geq3$ as hinted at above. 

For completeness, Table 1  below contains lower bounds for odd values of
$M$, $N$ and $K$ that can be found by using Yang states tensored with a
fixed element (not present in the pairs from which the pairing state is
built) whenever that is necessary.

\begin{cl} 
\label{cycl}
As a corollary
of Theorem
\ref{equivconst}, we have lower bounds for the optimal
eigenvalues $\Lambda^{M,N}_K$ defined in \eqref{eigproblem} that are listed in Table 1 below. Note
that the last two rows also give a result for $K=2k+1$
since $\Lambda^{M,N}_K=\Lambda^{M,N}_{N-K}$.
\end{cl}
\renewcommand{\arraystretch}{2.5}
\begin{table}[h!]\label{table1}
\begin{center}
\begin{tabular}{c|c|c|c}
$M$&$N$&$K$& lower bound for $\Lambda^{M,N}_K$\\
\hline
 $2m$ or $2m+1$ & $2n$ & $2k$ & $\frac{\tbinom{n}{k}}{\tbinom{N}{K}} \frac{\tbinom{m-n+k}{k}}{\tbinom{m}{k}}$\\
 $2m$ or $2m+1$ & $2n$ & $2k+1$ & $\frac{\tbinom{n-1}{k}}{\tbinom{N}{K}} \frac{\tbinom{m-n+k}{k}}{\tbinom{m-1}{k}}$\\
 $2m$ & $2n+1$ & $2k$ &$\frac{\tbinom{n}{k}}{\tbinom{N}{K}} \frac{\tbinom{m-1-n+k}{k}}{\tbinom{m-1}{k}}$\\
 $2m+1$ & $2n+1$ & $2k$ &$\frac{\tbinom{n}{k}}{\tbinom{N}{K}} \frac{\tbinom{m-n+k}{k}}{\tbinom{m}{k}}$\\
\end{tabular}
\end{center}
\end{table}

\section{Entanglement for fermionic mixed states}
\label{effms}
We now turn to results on bipartite entanglement for fermionic mixed states. There are a range of different entanglement measures that we can consider for mixed fermionic states \cite{B2}. 
In this section we consider two such measures: entanglement of formation and squashed entanglement.

By definition, a bipartite state $\rho_{12}$ state is {\em not entangled} if
and only if it is {\em separable}
which means that it is  in the closure of states of the form 
\[
\rho_{12} = \sum_{k=1}^n \nu_j\, \rho_1^j \otimes \rho_2^j,
\]
where the $\nu_j$ are positive and sum to $1$, and each
$\rho_\alpha^j$ is a density matrix on $\H_\alpha$.  

The 
{\em
entanglement of formation} $\ef$, introduced by Bennett
{\it et al.} \cite{B1,B2}, 
is defined in terms of the von Neumann entropy 
$S(\rho) = -\tr(\rho  \log \rho)$ by the formula
\[
\ef(\rho_{12}) = \inf \left\{ \sum_{j=1}^n \lambda_j 
S(\tr_2 \omega^j) \ :\ \rho_{12} = \sum_{j=1}^n\lambda_j
\omega^j \ \right\},
\]
where $\tr$ and $\tr_2$, respectively, are the traces over
the tensor product $\H_1\otimes \H_2$ and the partial 
trace over $\H_2 $ alone. 
The coefficients
$\lambda_j$ in the expansion are required to be positive and
sum to $1$, and each $\omega^j$ is a state on $\H_1\otimes
\H_2$, which, by the concavity of $S$, may be taken to be a
pure
state without affecting the value of the infimum.  Since the two partial
traces of a 
pure state have the same spectrum and hence the same entropies \cite{al},
$\ef(\rho_{12})$ is symmetric in $1$ and $2$.
It is
known that $\ef(\rho_{12}) =0$ if and only if $\rho_{12}$ is
separable; see \cite{BCY} for a discussion of this result in relation to
other measures of entanglement.

Any bipartite fermionic state will be entangled according to this
definition. Another definition that is appropriate for fermions is  to say
that
$\rho_{12} $ is {\em fermionic-separable} if and only if it is a convex
combination of projections onto 2-body Slater determinantal states
\cite{AFOV,KWCG}. 
Otherwise it is {\it fermionic entangled}. A number of authors
\cite{BCW,GM,IV,KWCG,
SCKLL,WS} have
proposed quantities that  measure the degree of fermionic entanglement.
One looks for a measure of
entanglement that is positive (sometimes called `faithful') on all
entangled states and zero on fermionic-separable states.

We first prove that Slater determinants uniquely minimize the  
usual `entanglement of formation', $\ef$,  and, therefore,
the excess of 
 $\ef$ over the Slater value is a faithful measure of fermionic
entanglement. This is
perhaps the first faithful quantification of fermionic entanglement that
uses conventional quantities, like $\ef$, which have an operational
meaning, unlike the `Slater rank', which is faithful by definition, but
which
is difficult to compute, is discontinuous, and does not have a clear
operational interpretation.

\begin{thm}\label{entform1}
Let $\rho_{12}$ be a bipartite fermionic state.
Then
\begin{equation}\label{lower}
\ef(\rho_{12}) \geq \ln(2),
\end{equation}
and there is equality if and only if $\rho_{12}$ is a
convex combination of 
pure-state Slater determinants; i.e., the state is fermionic separable. 
That is, the quantity
$$
\ef^A(\rho_{12}) := \ef(\rho_{12}) -\ln(2)
$$
is a faithful measure of fermionic entanglement. In particular, if
$\rho_{12}$ is the 2-particle reduced density matrix of an $N$-particle fermionic
state, then (\ref{lower}) is true and equality holds if and only if
the state $\rho_{12}$ is fermionic separable. 
\end{thm}

\begin{proof}
As discussed at the start of Section \ref{mutinf}, the largest eigenvalue of $\gamma_1$ cannot exceed $1/2$, and hence any fermionic
density matrix $\rho_{12}$ on $\H\otimes\H$ satisfies 
\[
 S(\gamma_1) \geq \ln 2,
\]
and this occurs exactly when $\gamma_1 $ is the reduced density
matrix of a 2-particle Slater determinant. This proves the result since $\ef$ is a convex combination of
such 1-body entropies, and if each is bounded below by $\ln 2 $ then so is
the convex combination.
\end{proof}

\begin{remark} Similar statements for pure states were made in \cite{GM}. 
 Were we computing the entropy using $\log_2$ in place of the natural
logarithm, the lower bound would be $1$. 
\end{remark}

Another faithful measure of entanglement is the {\em squashed
entanglement}, introduced by Tucci \cite{T} and studied by
Christandl and Winter \cite{CW1}. 
It is defined by 
\[
\es (\rho_{12}) =    \tfrac12 \inf_{\rho_{123}} \left\{ -S(\rho_{123}) -S(\rho_3) +
S(\rho_{13}) + S(\rho_{23}) \right\},
\]
where 3 refers to an additional Hilbert space and $\H = \H_1\otimes \H_2
\otimes \H_3$, and $\tr_3 \rho_{123} = \rho_{12}$. The infimum is taken over
all such extensions of $\rho_{12}$.  As a consequence of strong
subadditivity \cite{LR}, $\es(\rho_{12}) \geq 0$.

The squashed entanglement is a
faithful measure of entanglement, meaning that $\es(\rho_{12} ) =0$  if and
only if $\rho_{12}$ is separable (in the usual non-fermionic
sense) \cite{BCY}.
It is less than or
equal to the entanglement of formation,
and it is claimed to measure only quantum mechanical correlations.
\begin{conjecture}
\label{c4}
Convex combinations of Slater determinantal states uniquely minimize the
squashed entanglement, as they do for $\ef$.  
\end{conjecture}

\textbf{Question:} If $\gamma_{2}$ is the 2-particle reduced density matrix of an
$N$-particle fermionic state, is $\es (\gamma_2) $ greater than or equal
to the squashed entanglement of an $N$-particle Slater determinant? If so,
is the difference a faithful measure of fermionic entanglement? 

Theorem \ref{entform1} shows this to
be the case for entanglement of formation, and the first step in the
proof was  to
compute the $\ef$ of a two-particle density matrix  of a Slater. This
number turned out to
be independent of $N$.  The situation is different for $\es$, for we cannot
compute $\es$ for  a Slater determinant,  but  the following
theorem, discovered by Christandl, Schuch and Winter \cite{CSW},
definitely shows that there must be an $N$ dependence.

\begin{thm}[Squashed entanglement for Slaters] \label{sqthm}
Let $\gamma_2$ be the
2-particle reduced density matrix of an
$N$-particle Slater determinant. Then 
\begin{equation} \label{squash} 
 \es (\gamma_2) \leq   
 \begin{cases} 
\ln \frac{N+2}{N}\ \ \ \ \ {\rm if \ } N \ {\rm
is\ even}\\
\phantom{xx}\\
\tfrac12 \ln \frac{N+3}{N-1} \ \ \ \ \ {\rm if \ } N \ {\rm is\  odd}
 \end{cases}.
\end{equation}
\end{thm}
This shows that the squashed entanglement can be {\em much  smaller} than
the entanglement of formation. 
\begin{conjecture}
\label{c5}
Inequality \eqref{squash} is 
actually an equality. Moreover, it gives the lowest possible $\es$
among all
fermionic 2-particle reduced density matrices.
\end{conjecture}

The fact that $\es$ is so small for a large $N$ fermionic state indicates that the squashed entanglement
may already be a good measure of fermionic entanglement, without any further subtraction for large $N$.

For Slaters the entanglement of formation and the squashed entanglement are
very different, but the following calculation for the Yang state shows that
this need not always be so. 

\begin {thm}[Entropy and entanglement of the pairing state] \label{pairing}
Let $M=2m$ and $N=2n$. Let  $\gamma_2$ be the two-particle reduced density
matrix of the 
N-particle pairing state  built on  $\H= \C^M$. 
Its fermionic entanglement of formation is 
\[
\ef(\gamma_2) - \ln(2)  =   \frac{(m-n) }{(2n-1)(m-1)}[\ln(m) - \ln(2)],
\]
and the squashed entanglement is bounded by
\[
\es(\gamma_2)  \leq  \frac{(m-n) }{(2n-1)(m-1)}\ln(m) + \left(1 -
\frac{(m-n) }{(2n-1)(m-1)}\right)  \ln\left(\tfrac{m+1}{m}\right)\
.
\]
\end {thm}
Note that for $M \gg N \gg 1$, 
the leading term is $\ef(\gamma_2) \asymp  \tfrac{\ln M}{N} $, and the
bound on $\es(\gamma_2)$ 
is of the same order and is, presumably, close to optimal.

\section{Proofs of Theorems}
\label{proofthms}
While Theorem \ref{mutual2} refers to fermionic states, we shall deduce
it from the following theorem in which {\em no assumption about statistics} is
made.
It is a quantitative version of {\it subadditivity of the von Neumann
entropy}, for {\em general} bipartite (and $N$-partite) states, 
which we have not seen
before and might be useful in other cases. The method of proof of this
theorem also yields quantitative remainder terms for other entropy
inequalities -- which were discussed in \cite{clb}.

\begin{thm}[Quantitative subadditivity]\label{quant}
Let $\rho_{12}$ be a density matrix on a bipartite Hilbert space $\H_1\otimes\H_2$, and let $\rho_1$ and $\rho_2$ be its reduced density matrices on $\H_1$ and $\H_2$, respectively. 
\begin{equation}\label{1state}
S(\rho_{12}) - S(\rho_1) -S(\rho_2) \leq 2\ln\left(1 - \tfrac12 \tr
\left[\sqrt{\rho_{12}}-\sqrt{\rho_1\otimes \rho_2}\right]^2\right).
\end{equation}
In particular, $S(\rho_1) +S(\rho_2) - S(\rho_{12}) \geq 0$ with equality if and only if 
$\rho_{12} = \rho_1\otimes \rho_2$. 

More generally, with an obvious notation, if 
$\rho_{1\cdots N}$ is a density matrix on $\H_1\otimes \cdots \otimes \H_N$
then  

\begin{equation} \label{nstate}
S(\rho_{1\cdots N}) -\sum_{j=1}^N S(\rho_j)
\leq 2\ln\left(1 - \tfrac12\tr 
\left[ \sqrt{\rho_{1\cdots N} } - \sqrt{\rho_1 \otimes \cdots \otimes
\rho_N}\right]^2\right).
\end{equation}
\end{thm}

\begin{proof}
Recall the Peierls--Bogoliubov inequality: If $H$ and $A$ are self adjoint
operators and $\tr e^{-H} =1$, then 
$$\tr\left(e^{-H+A}\right) \geq e^{\tr Ae^{-H}}\ .$$
To prove (\ref{1state}), apply this with 
$$H = - \log \rho_{12} \qquad{\rm and}\qquad A = \tfrac12(\log \rho_1 +
\log\rho_2 - \log\rho_{12})\ .$$
Then with $\Delta:= \tfrac12 (S(\rho_{12}) - S(\rho_1) -S(\rho_2))$, by the
Peierls--Bogoliubov
inequality and the Golden--Thompson inequality,
\begin{eqnarray}
e^\Delta &=& \exp \left [\tr\rho_{12}\tfrac12 (\log \rho_1 + \log\rho_2 -
\log\rho_{12})\right]\nonumber\\
&\leq &  \tr \exp\left[ \tfrac12 (\log \rho_{12} +  \log(\rho_1\otimes
\rho_2))\right]\nonumber\\
&\leq &  \tr \exp\left[ \tfrac12 \log \rho_{12}  \right]\exp\left[\tfrac12 
\log(\rho_1\otimes \rho_2)\right]\nonumber\\
&=& \tr \left[ \rho_{12}^{1/2} (\rho_1\otimes \rho_2)^{1/2} \right]\ .\nonumber
\end{eqnarray}
Since
\begin{equation}\label{proj4}
\tr \left[ \rho_{12}^{1/2} (\rho_1\otimes \rho_2)^{1/2} \right]  = 
\left(1 - 
\tfrac12\tr \left[ \rho_{12}^{1/2} -(\rho_1\otimes \rho_2)^{1/2} \right]^2\right)\ ,
\end{equation}
this proves (\ref{1state}). An obvious adaptation proves  (\ref{nstate}). 
\end{proof}

\begin{proof}[{\bf Proof of Theorem~\ref{mutual2}}]
By the hypothesis on $\gamma_2$, all of its eigenfunctions with non-zero 
eigenvalues are antisymmetric, and $\rho_1 = \rho_2=\gamma_1$ in Theorem \ref{quant}. Therefore,
$$\tr\left(
\sqrt{\gamma_{2}} \sqrt{\gamma_1\otimes \gamma_1}\right) = \tr\left(
\sqrt{\gamma_{2}} \left[ P_{{\rm fer}}\sqrt{\gamma_1\otimes \gamma_1}P_{{\rm
fer}}\right]\right)\ ,$$
where  $P_{{\rm fer}}$ is the orthogonal projection on the antisymmetric
subspace of $\H\otimes \H$, and then by the Cauchy--Schwarz inequality,
\begin{equation}\label{proj3}
\tr\left(\sqrt{\gamma_{2}} \sqrt{\gamma_1\otimes \gamma_1}\right)    = \tr\left(P_{{\rm fer}}\sqrt{\gamma_{2}} P_{{\rm fer}}\sqrt{\gamma_1\otimes \gamma_1}\right) 
\leq\left(\tr \left[ P_{{\rm fer}}\gamma_1\otimes \gamma_1P_{{\rm
fer}}\right]\right)^{1/2}\ .
\end{equation}
Let $\sum_{j} \lambda_j |u_j\rangle\langle u_j|$ denote
the spectral decomposition of $\gamma_1$.   Then
\begin{equation}\label{proj1}
P_{{\rm fer}}\gamma_1\otimes \gamma_1P_{{\rm fer}}  = \tfrac12 \sum_{i < j} 
\lambda_i\lambda_j |u_i\wedge u_j\rangle \langle u_i\wedge u_j|,
\end{equation}
where $u_i\wedge u_j = u_i\otimes u_j - u_j\otimes u_i$ is a vector of
length $\sqrt{2}$, hence the factor of $\tfrac12$.
Thus,
$$\tr\left[ P_{{\rm fer}}\gamma_1\otimes \gamma_1P_{{\rm fer}}\right] = \sum_{i
< j}\lambda_i\lambda_j 
=  \tfrac12(1 - \tr\gamma_1^2)\ .$$
Combining this with (\ref{1state}), (\ref{proj4}) and  (\ref{proj3}), we
obtain
\begin{equation}\label{subadd2}
S(\gamma_2) \leq 2S(\gamma_1) - \ln\left(\frac{2}{1 - \tr\gamma_1^2}\right)\ .
\end{equation}

The fact that there is equality when $\gamma_{2}$ is the reduced $2$-particle density matrix of an $N$-particle Slater has
been discussed below the statement of the theorem.

Moreover, whenever there is equality in (\ref{subadd2}) , there must be
equality in (\ref{proj3}), in which case for some constant
$C$, $\sqrt{\gamma_{2}} = CP_{{\rm fer}}\sqrt{\gamma_1\otimes \gamma_1}P_{{\rm
fer}}$, or, what is the same thing by (\ref{proj1})
$$\gamma_2  = \frac{1}{2}C^2 \sum_{i< j} \lambda_i\lambda_j  |u_i\wedge u_j\rangle
\langle u_i\wedge u_j|\ .$$
Taking the partial trace $\tr_2$ of both sides, we obtain
${\displaystyle \gamma_1 = \frac12 C^2 \sum_j \lambda_j(1-\lambda_j)
|u_j\rangle\langle u_j|}$, from which we conclude
that $\frac{1}{2}C^2(1-\lambda_j) = 1$ for each $j$. This means that $\gamma_1$ is a
normalized projection, and that $\gamma_2$
is the $2$-particle reduced density matrix of a pure Slater determinant.

The inequality \eqref{subedn} follows from \eqref{nstate} in the same way
that  \eqref{subadd1} follows from (\ref{1state}). 
The equation \eqref{power} for $e_N$ in terms of power sums 
(using $p_1=\tr \gamma_1 =1$) is  well known.
\end{proof}

\begin{proof}[{\bf Proof of Theorem~\ref{sqthm}}]
Let $\Psi$ be an $N$-particle Slater determinant. We choose 
$\rho_{123} $ to be the $K$-particle reduced density matrix
with $K\geq 2$. Thus, $\H_3$ is the $(K-2)$-particle fermionic
space, which has dimension $\tbinom{N}{K-2}$. We compute as follows:
$$S(\rho_{123}) = \ln \tbinom{N}{K}, \ S(\rho_3) =  \ln \tbinom{N}{K-2}, \
S(\rho_{13}) = S(\rho_{23}) = \ln \tbinom{N}{K-1}.
$$
Thus, 
$$
-S(\rho_{123}) - S(\rho_3) + S(\rho_{13}) + S(\rho_{23}) = \ln [K/(N-K+1)]+ \ln [ (N-K+2)/ (K-1)],
$$
and the theorem is proved by choosing $K=(N+2)/2$ for $N$ even and $K= (N+1)/2$ (or $K = (N+3)/2$)
for $N$ odd. 
\end{proof}
\bigskip
\bigskip

\begin{proof}[\bf Proof of Theorem \ref{pairing}]
Proposition \ref{Ys} and its proof can be used to compute $\ef(\gamma_2)$ and to
estimate $\es(\gamma_2)$.
Let $\alpha$ denote the completely antisymmetric state on $\H\otimes \H$;
i.e., $\alpha = (m(2m-1))^{-1}P_{\H\wedge \H}$. 
Then we may write (\ref{form3}) as
\[
\gamma_2 =   \frac{(m-n) }{(2n-1)(m-1)} |\chi\rangle\langle \chi| +  \left(
1 - \frac{(m-n) }{(2n-1)(m-1)} \right) \alpha\ .
\]
In any decomposition $\gamma_2 = \sum_{j=1}^n\lambda_j \omega^j$, one of
the $\omega_j$ must be $|\chi\rangle\langle \chi|$
and the corresponding $\lambda_k$ must be at least $(m-n)/(2n-1)(m-1)$.   Since $S(\tr_1|\chi\rangle\langle \chi|) = \ln m$,
it follows from these computations and Theorem~\ref{entform1} that
$$\ef(\gamma_2) - \ln(2)  =   \frac{(m-n) }{(2n-1)(m-1)}[\ln(m) - \ln(2)] \
.$$

Since $\es(|\chi\rangle\langle \chi|) = \ln(m)$ it follows from the
convexity of squashed entanglement \cite{CW1} and Theorem~\ref{sqthm} that
$$\es(\gamma_2)  \leq  \frac{(m-n) }{(2n-1)(m-1)}\ln(m) + \left(1 -
\frac{(m-n) }{(2n-1)(m-1)}\right)
 \ln\left(\tfrac{m+1}{m}\right)\ .$$

Note that for $M \gg N \gg 1$, 
the leading term is $\ef(\gamma_2) \asymp  \tfrac{\ln M}{N} $, and the
upper bound on $\es(\gamma_2)$ 
is of the same order and is, presumably,  close to optimal.  
\end{proof}

\medskip
\noindent {\bf Acknowledgements:} We thank U.~Marzolino and J.~Schliemann
for helpful correspondence.  We thank the  U.S.
National Science
Foundation
for support  by grants DMS 1501007. (E.A.C.) and PHY 0965859  and PHY-1265118
(E.H.L.) and the European Research Council for support by ERC Grant Agreement Nos. 321029 and 337603 (R.R.).

\end{document}